\documentclass[11pt]{article}
\usepackage{authblk}

\usepackage[latin1]{inputenc}
\usepackage[T1]{fontenc}

\usepackage{authblk} 

\usepackage{amsmath}
\usepackage{amsfonts}
\usepackage{amssymb}

\usepackage[normalem]{ulem} 

\usepackage{amsmath,xcolor,ulem}
\usepackage{amsfonts}
\usepackage{amssymb,todonotes}
\usepackage{amsthm}
\usepackage{graphicx}
\usepackage{marginnote}

\usepackage[colorlinks=true, allcolors=blue]{hyperref}

\usepackage[top=2.8cm,bottom=2.8cm,left=2.5cm,right=2.5cm]{geometry}
\usepackage{float}
\usepackage[algoruled]{algorithm2e}
\usepackage[font=footnotesize,width=\linewidth]{caption} 

\usepackage{graphicx}
\usepackage{amssymb}
\usepackage{amsthm}
\usepackage{amsmath}
\usepackage{lineno}
\usepackage{hyperref}
\usepackage{verbatim}

\newtheorem{theorem}{Theorem}[section]
\newtheorem{lemma}[theorem]{Lemma}
\theoremstyle{definition}

\theoremstyle{remark}
\newtheorem{remark}[theorem]{Remark}

\numberwithin{equation}{section}

\DeclareMathOperator{\interior}{int}

\usepackage{graphicx, pdflscape}
\usepackage{amsmath, hyperref}
 \usepackage{amssymb} 
\usepackage{relsize, float, lineno}

 \usepackage{xcolor} 
\usepackage{subcaption}
\allowdisplaybreaks

\title{A Generalized Second-Order Positivity-Preserving Numerical Method  for Non-Autonomous Dynamical Systems \\ with Applications}

\author{Manh Tuan Hoang\footnote{ \href{mailto:tuanhm16@fe.edu.vn}{tuanhm16@fe.edu.vn}} ,
Matthias Ehrhardt\footnote{ \href{mailto:ehrhardt@uni-wuppertal.de}{ehrhardt@uni-wuppertal.de}} 
}

\affil{Department of Mathematics, FPT University, Hoa Lac Hi-Tech Park, \\ Km29 Thang Long Blvd, Hanoi, Viet Nam}
\affil{IMACM, School of Mathematics and Natural Sciences, \\ University of Wuppertal, Germany}

\begin{document}
\maketitle

\begin{tikzpicture}[remember picture,overlay]
	\node[anchor=north east,inner sep=20pt] at (current page.north east)
	{\includegraphics[scale=0.2]{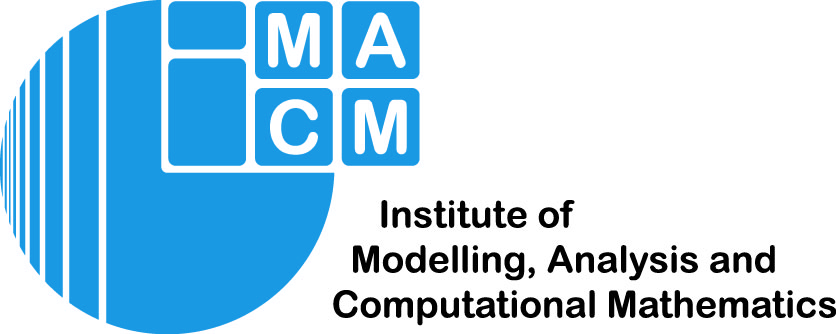}};
\end{tikzpicture}

\begin{abstract}
In this work, we propose a generalized, second-order, nonstandard finite difference (NSFD) method for non-autonomous dynamical systems. The proposed method combines the NSFD framework with a new non-local approximation of the right-hand side function. This method achieves second-order convergence and unconditionally preserves the positivity of solutions for all step sizes. Especially, it avoids the restrictive conditions required by many existing positivity-preserving, second-order NSFD methods. 
The method is easy to implement and computationally efficient. Numerical experiments, including an improved NSFD scheme for an SIR epidemic model, confirm the theoretical results. Additionally, we demonstrate the method's applicability to nonlinear partial differential equations and boundary value problems with positive solutions, showcasing its versatility in real-world modeling.
\end{abstract}

\begin{minipage}{0.9\linewidth}

\medskip


\end{minipage}

\section{Introduction}\label{Sec1}
Various important processes and phenomena in real-world situations can be modeled mathematically  by non-autonomous dynamical systems of the form:
\begin{equation}\label{eq:1}
     y'(t) = F(t, y(t)), \quad y(0) = y_0 \in \mathbb{R}^n,
\end{equation}
where
$y(t) = [y_1(t),\,y_2(t),\ldots,y_n(t)]^\top\colon\mathbb{R} \to \mathbb{R}^n$ is an unknown function that must be determined as the solution;
$F(t, y) = [F_1(t, y),\,F_2(t, y),\ldots, F_n(t, y)]^\top\colon \mathbb{R}\times\mathbb{R}^n \to \mathbb{R}^n$ satisfies conditions that guarantee the existence and uniqueness of solutions to \eqref{eq:1} \cite{Allen, Hirsch, Mattheij, Smith}.

For these dynamical systems, the positivity of the solutions can be considered as the most common and important characteristic \cite{Allen, Mattheij, Smith}. 
This characteristic can be easily investigated by a simple necessary and sufficient condition (see \cite[Lemma 1]{Horvath1} and \cite[Proposition B.7]{Smith}): 
The solution of \eqref{eq:1} admits the set $\mathbb{R}^n_+ = \{(y_1,\,y_2,\ldots,y_n) \in \mathbb{R}^n|y_1 \geq 0,\,\,\,y_2 \geq 0,\ldots,y_n \geq 0\}$ as a positively invariant set if and only if
\begin{equation}\label{eq:2}
F_i(t, y)|_{y_i = 0} := F_i(t, y_1,\ldots,y_{i -1},\,0,\,y_{i + 1},\ldots y_n) \geq 0
\end{equation} 
for $i = 1, 2, \ldots, n$ and $(t, y)\in\mathbb{R}_+ \times \mathbb{R}^n_+$.

Numerical methods that preserve the positivity of solutions to \eqref{eq:1} are essential in both theory and practice (see, for instance, \cite{Horvath1, Horvath2, Mickens1, Mickens2, Mickens3, Mickens4, Mickens5}). 
Nonstandard finite difference (NSFD) schemes, which were first introduced by Mickens in the 1980s \cite{Mickens1, Mickens2, Mickens3, Mickens4, Mickens5}, have become an efficient approach to the positivity-preserving problem of numerical methods. 
Specifically, NSFD schemes have the ability to preserve not only the positivity of the solutions but also other qualitative dynamical properties for all step sizes \cite{Mickens1, Mickens2, Mickens3, Mickens4, Mickens5, Patidar1, Patidar2}. 
However, NSFD schemes typically achieve only first-order accuracy.
For this reason, high-order NSFD schemes for dynamical systems governed by nonlinear ordinary differential equations (ODEs), mainly second-order schemes, have been intensively studied in recent years (see \cite{Alalhareth1, Alalhareth2, Conte, Hoang1, Hoang2, Hoang3, HoangMatthias1, HoangMatthias2, Kojouharov1} and references therein). 
These NSFD schemes are derived from methodology of Mickens \cite{Mickens1, Mickens2, Mickens3, Mickens4, Mickens5} with a non-local approximation for the right-hand side functions and renormalization of the denominator functions.

Second-order NSFD methods have been developed for one-dimensional dynamical systems (see, e.g., \cite{Hoang1, Hoang2, Hoang3, Kojouharov1}). For multi-dimensional dynamical systems, Alalhareth et al.\ \cite{Alalhareth1} have been developed the approach used in \cite{Wood} to construct second-order modified positive and elementary stable (SOPESN) NSFD methods for $n$-dimensional autonomous differential equations. 
These SOPESN methods were subsequently employed in \cite{Alalhareth3} to numerically solve a mathematical model of nutrient recycling and dormancy in a chemostat. 
In a recent work \cite{HoangMatthias2}, the authors 
have used a nonlocal approximation with right-hand side function weights and nonstandard denominator functions to construct a second-order and dynamically consistent NSFD method for a general Rosenzweig-MacArthur predator-prey model. 
Recently, Conte et al.\ \cite{Conte} have derived a general procedure to obtain unconditionally positive second-order NSFD methods.
Furthermore, adding parameters to these schemes for each particular problem allows one to determine the optimal parameter values to guarantee positivity, elementary stability, and minimization of the local truncation error.

Inspired by the importance of positivity-preserving numerical methods, this work proposes a  straightforward method for constructing second-order positivity-preserving numerical methods for system \eqref{eq:1}. 
Throughout this paper, we will consider the system \eqref{eq:1} under the condition \eqref{eq:2}.

It is not difficult to show that for each $1 \leq i \leq n$:
\begin{itemize}
\item if $F_i(t, y)|_{y_i = 0} = 0$ and $y_i(0) = 0$, then $y_i(t) = 0$ for all $t \geq 0$ is the unique solution;
\item if $F_i(t, y)|_{y_i = 0} = 0$ and $y_i(0) > 0$, then $y_i(t) > 0$ for all $t > 0$;
\item if $F_i(t, y)|_{y_i = 0} > 0$ and $y_i(0) > 0$, then $y_i(t) > 0$ for all $t > 0$;
\item if $F_i(t, y)|_{y_i = 0} > 0$ and $y_i(0) = 0$, then there exists $t_0 > 0$ such that $y_i(t) > 0$ for all $t \geq t_0$.
\end{itemize}
Consequently, without loss of generality, we can consider \eqref{eq:1} with strictly positive solutions, that is for each $i$
\begin{equation}\label{eq:3}
y_i(0) > 0 \Longrightarrow y_i(t) > 0\quad \text{for} \quad t > 0.
\end{equation}
In other words, the system \eqref{eq:1} admits the interior $\interior(\mathbb{R}^n_+)$ of $\mathbb{R}^n_+$ as a positively invariant set. Therefore, our goal is to develop second-order numerical methods that possess the property
\begin{equation}\label{eq:4}
       y_i(0) > 0 \Longrightarrow y_i^k > 0 \quad \text{for all}     \quad k = 1, 2, \ldots \quad \text{and} \quad \Delta t > 0,
\end{equation}
where $\Delta t > 0$ is the step size and $y_i^k$ is the intended approximation of $y_i(t^k)$ with $t^k = k\Delta t$ for $k = 1, 2, \ldots$.

Based on a representation theorem \cite[Theorem 10]{Cresson2}, it is important to note that the system~\eqref{eq:1} can be represented in the form
\begin{equation}\label{eq:5}
    y_i'(t) = f_i(t, y) - y_i \,g_i(t, y), \quad i = 1, 2, \ldots, n
\end{equation}
where $f_i$ and $g_i$ are two functions from $\mathbb{R}_+ \times \interior\Big(\mathbb{R}^n_+\Big) \to \mathbb{R}_+$.

From now on, we will work with \eqref{eq:5} instead of \eqref{eq:1}.
Using the approaches used in \cite{Alalhareth1, Alalhareth2} and \cite{Conte}, one can obtain second-order positivity-preserving schemes for \eqref{eq:1}. 
However, as will be discussed in Section~\ref{Sec2}, the resulting NSFD schemes require a strict and indispensable condition (Condition~\eqref{eq:7}), which limits their applicability in computations. 
In contrast, the NSFD method proposed in this work relaxes this condition. As a result, its computational implementation is straightforward. 
It is well-known that Runge-Kutta methods only guarantee the positivity preserving property in many situations if the step size is smaller than a positivity step size threshold (see, e.g.,  \cite{Horvath1, Horvath2}). 
However, the constructed NSFD method is positivity-preserving regardless of the chosen step size. In other words, it is unconditionally positive.

The paper is organized as follows.
In Section \ref{Sec2}, we apply the well-known approaches proposed in \cite{Alalhareth1, Alalhareth2} and \cite{Conte} to obtain second-order positivity-preserving NSFD schemes for \eqref{eq:5}, thereby identifying a strict and  indispensable condition (Condition \eqref{eq:7}) imposed on the resulting NSFD schemes.
In Section \ref{Sec3}, we construct and analyze a generalized second-order positivity-preserving NSFD method for which the condition \eqref{eq:7} is relaxed.
In Section \ref{Sec4}, we conduct a set of numerical experiments to support and illustrate the theoretical results. In these experiments, we consider a modified Susceptible-Infected-Removed (SIR) model \cite{Bailey} as a test problem. An important consequence is that the dynamically consistent NSFD scheme for the SIR model, constructed  very recently in \cite{Lemos-Silva}, is improved.
Finally, in Sections \ref{Sec5} and \ref{Sec6}, we apply the constructed NSFD method to solve some classes of partial differential equations (PDEs) and boundary value problems (BVPs) with positive solutions.
The last section includes some concluding remarks and discussions.

\section{NSFD Schemes Based on Well-Known Approaches}\label{Sec2}
In this section, we apply the NSFD methods constructed in \cite{Alalhareth1, Alalhareth2} and \cite{Conte} to obtain second-order NSFD positivity-preserving NSFD schemes for \eqref{eq:5}.

First, applying the approach in \cite{Alalhareth2} leads to the following NSFD scheme for \eqref{eq:5}:
\begin{equation}\label{eq:NSFD1}
\dfrac{y_i^{k + 1} - y_i^k}{\phi_i(\Delta t, t^k, y^k)} = f_i(t^k, y_i^k) - y_i^{k + 1}g_i(t^k, y_i^k),
\end{equation}
where $\phi_i(.)$ is a function satisfying
\begin{equation}\label{eq:6}
\begin{split}
&\phi_i(\Delta_t, t, y) > 0 \quad \text{for all} \quad \Delta t > 0, \quad (t, y) \in \mathbb{R}_+ \times \interior(\mathbb{R}^n_+),\\
&\phi_i(\Delta_t, t, y) = \Delta t + \mathcal{O}(\Delta t^2) \quad \text{as} \quad \Delta t \to 0.
\end{split}
\end{equation}
The system \eqref{eq:NSFD1} can be written in the form
\begin{equation}\label{eq:NSFD1aR}
y_i^{k + 1} = \dfrac{y_i^k + \phi_i\,f_i(t^k, y^k)}{1 + \phi_i\,g_i(t^k, y^k)}.
\end{equation}
This implies that \eqref{eq:NSFD1} preserves the positivity of the solutions for all $\Delta t > 0$. 
Note that first-order NSFD schemes  for a general class of two ODEs constructed \cite{Cresson1} and for a $n$-dimensional productive-destructive systems \cite{Wood1} can be derived from \eqref{eq:NSFD1aR} with $\phi_i = \phi$ for all $i = 1, 2, \ldots n$.

A condition ensuring the second-order accuracy of \eqref{eq:NSFD1} is determined via Taylor's expansion theorem as follows (see \cite{Alalhareth2}).
\begin{lemma}\label{Lemma1}
Assume that the denominator functions $\phi_i$ ($i = 1, 2, \ldots, n$) satisfy \eqref{eq:6}. 
Then, the truncated error of \eqref{eq:NSFD1} is $\mathcal{O}(\Delta t^3)$ as $\Delta t \to 0$ whenever
\begin{equation}\label{eq:NSFD1a}
\dfrac{\partial^2 \phi_i}{\partial \Delta t^2}(0,t,y) = 2g_i(t, y) + \dfrac{1}{F_i(t, y)}\bigg(\dfrac{\partial F_i}{\partial t}(t, y) + \sum_{j = 1}^n\dfrac{\partial F_i}{\partial y_j}(t, y)\,F_j(t, y)\bigg)
\end{equation}
for all $t \geq 0$, $y \in \interior(\mathbb{R}^n_+)$ such that $F_i(t, y) \ne 0$, where $F_i(t, y) = f_i(t, y) + y_i\,g(t, y)$ is the right-hand side function of the $i$-th equation of \eqref{eq:1}.
\end{lemma}

Next, we apply the approach used in \cite{Alalhareth1} (as well as in \cite{Alalhareth2}) to obtain a second-order and positive NSFD scheme for \eqref{eq:1}. The resulting NSFD scheme is given by
\begin{equation}\label{eq:NSFD2}
\dfrac{y_i^{k + 1} - y_i^k}{\phi_i(\Delta t, t^k, y^k)} = w_i^kF_i(t^k, y^k),
\end{equation} 
where
\begin{equation*}
w_i^k :=
\begin{cases}
&1, \quad \text{if} \quad F_i(t^k, y^k) \geq 0,\\
&\dfrac{y_i^{k + 1}}{y_i^k}, \quad  \text{if} \quad F_i(t^k, y^k) < 0,
\end{cases}
\end{equation*}
and $\phi_i(.)$ are functions satisfying \eqref{eq:6}.

The following result is proven based on the proof of  \cite[Theorem 3.2.1]{Alalhareth2} (see also \cite{Alalhareth1}).
\begin{lemma}\label{Lemma2}
Assume that the denominator functions $\phi_i$ ($i = 1, 2, \ldots, n$) satisfy \eqref{eq:6}. Then, the truncated error of \eqref{eq:NSFD1} is $\mathcal{O}(\Delta t^3)$ as $\Delta t \to 0$ whenever
\begin{equation}\label{eq:NSFD2a}
\dfrac{\partial^2 \phi_i}{\partial \Delta t^2}(0,t,y) =\\
\begin{cases}
&\dfrac{1}{F_i(t, y)}\bigg(\dfrac{\partial F_i}{\partial t}(t, y) + \mathlarger{\mathlarger{\sum_{j = 1}}}^n\dfrac{\partial F_i}{\partial y_j}(t, y)F_j(t, y)\bigg) \quad \text{if} \quad F_i(t, y) \geq 0,\\[.2cm]
&2\dfrac{F_i(t, y)}{y_i^2} - \dfrac{1}{F_i(t, y)y_i^2}\mathlarger{\mathlarger{\sum}}_{j = 1}^n\bigg(\dfrac{\partial F_i}{\partial t}(t, y) + \mathlarger{\mathlarger{\sum}}_{j = 1}^n\dfrac{\partial F_i}{\partial y_j}(t, y)F_j(t, y)\bigg)\,\,\, \text{if} \quad F_i(t, y) < 0
\end{cases}
\end{equation}
for all $t \geq 0$, $y \in \interior(\mathbb{R}^n_+)$ such that $F_i(t, y) \ne 0$, where $F_i(t, y)$ is the right-hand side function of the $i$-th equation of \eqref{eq:1}.
\end{lemma}

We now construct another positivity-preserving and second-order NSFD scheme for \eqref{eq:1}, based on the $\alpha$-NSFD method recently formulated in \cite{Conte}. 
The resulting NSFD scheme is given by
\begin{equation}\label{eq:NSFD3}
\dfrac{y_i^{k+1} - y_i^k}{\phi_i(\Delta t, t^k, y^k)} = F_i(t^k, y^k) - \alpha^i \dfrac{y_i^{k + 1} - y_i^k}{y_i^k}F_{i-}(t^k, y^k),
\end{equation}
where the right-side functions $F_i = F_{i+} + F_{i-}$ are split into a positive $F_+$ term and a negative $F_{i-}$ term;  $\alpha^i$ are non-negative real numbers for $i = 1, 2, \ldots, n$.

Applying \eqref{eq:NSFD3} to \eqref{eq:5} yields
\begin{equation}\label{eq:NSFD3aR}
\dfrac{y_i^{k+1} - y_i^k}{\phi_i(\Delta t, t^k, y^k)} = F_i(t^k, y^k) - \alpha^i \dfrac{y_i^{k + 1} - y_i^k}{y_i^k}y_i^kg_i(t^k, y^k).
\end{equation}
Note that \eqref{eq:NSFD3aR} reduces to \eqref{eq:NSFD1} if $\alpha^i = 1$. Furthermore, we deduce from \cite[Theorem 3]{Conte} that \eqref{eq:NSFD3a} preserves the positivity of the solutions of \eqref{eq:5} if $\alpha^i \geq F_i(t^k, y^k)/(y_i^kg_i(t^k, y^k))$ for all $t_k \geq 0$ and $y^k \in \interior(\mathbb{R}^n_+)$. 
A condition for \eqref{eq:1} to be second-order accurate was given in \cite[Theorem 4]{Conte}. Based on this, we obtain the following lemma.
\begin{lemma}\label{Lemma3}
Assume that the denominator functions $\phi_i$ ($i = 1, 2, \ldots, n$) satisfy \eqref{eq:6}. Then, the truncated error of \eqref{eq:NSFD1} is $\mathcal{O}(\Delta t^3)$ as $\Delta t \to 0$ provided that
\begin{equation}\label{eq:NSFD3a}
\dfrac{\partial \phi_i}{\partial \Delta t}(0,t,y) = 2\alpha^i \,g_i(t, y) + \dfrac{1}{F_i(t, y)}\bigg(\dfrac{\partial F_i}{\partial t}(t, y) + \sum_{j = 1}^n\dfrac{\partial F_i}{\partial y_j}(t, y)\,F_j(t, y)\bigg)
\end{equation}
for all $t \geq 0$, $y \in \interior(\mathbb{R}^n_+)$ such that $F_i(t, y) \ne 0$, where $F_i(t, y)$ is the right-hand side function of the $i$-th equation of \eqref{eq:1}
\end{lemma}

\begin{remark}
Consistency is a local property of one-step schemes, such as the NSFD schemes \eqref{eq:NSFD1}, \eqref{eq:NSFD2}, and \eqref{eq:NSFD3}. Using the well-known result that the convergence order can follow from the consistency order \cite{Ascher},  we obtain the NSFD schemes convergence of order $2$ from the second-order consistent ones.
\end{remark}
\begin{remark}
Lemmas \ref{Lemma1}--\ref{Lemma3} provide the conditions for the NSFD schemes to be convergent of order $2$.
However, it is easy to see a strict and indispensable condition for the NSFD schemes  \eqref{eq:NSFD1}, \eqref{eq:NSFD2} and \eqref{eq:NSFD3} is
\begin{equation}\label{eq:7}
F_i(t^k, y^k) \ne 0 \quad \text{for all} \quad k \geq 0 \quad \text{and} \quad i = 1, 2, \ldots, n. 
\end{equation}
This condition can be removed for 1-D dynamical systems \cite{Hoang1, Hoang2, Hoang3, Kojouharov1}; however, it limits the applicability of the corresponding NSFD schemes for computing solutions to multi-dimensional dynamical systems. 
To illustrate this, we consider the following simple system
\begin{equation*}\label{eq:8}
\begin{split}
&y_1' = F_1(t, y_1, y_2) := t^2(y_1 - a)^2 + t^4(y_2 - b)^4,\\
&y_2' = F_2(t, y_1, y_2) := t^4(y_1 - c)^2 + t^2(y_2 - d)^2,
\end{split}
\end{equation*}
subject to the initial data: $y_1(0) > 0$ and $y_2(0) > 0$, where $a$, $b$, $c$, $d$ are positive real numbers. 
If at a certain iteration $k$ ($k \geq 0$) we obtain $y_1^k = a, y_2^k = b$, then $F_1(t^k, y_1^k, y_2^k) = 0$; consequently, the condition \eqref{eq:7} is violated. 
The same can be said if $y_1^k = c, y_2^k = d$ for some $k \geq 0$.
For autonomous dynamical systems, \eqref{eq:NSFD1a}, \eqref{eq:NSFD2a} and \eqref{eq:NSFD3a} are not satisfied if there exists an approximation $y^k$ belonging to the nullclines of the dynamical systems under consideration.

The NSFD scheme \eqref{eq:NSFD2} requires determining the sign of the right-hand side function at each iteration step to choose $w_i^k$. 
Similarly, the NSFD scheme \eqref{eq:NSFD3a} requires choosing the value of $\alpha^i$ at each iteration step.
\end{remark}

\section{Construction of New Second-Order Positivity-Preserving NSFD Method}\label{Sec3}
In this section, we will construct a generalized, second-order, positivity-preserving NSFD method that relaxes the condition \eqref{eq:7}.

For any function $u$ from $\interior(\mathbb{R}^m_+)$ to $\mathbb{R}_+$, we define
\begin{equation*}
\mathcal{D}(u) = \{(u_+, u_-)|u_+, u_-: \interior(\mathbb{R}^m_+) \to \mathbb{R}; u_+, u_- \geq 0, u_+ - u_- = u\}.
\end{equation*}
It is easy to see that $\mathcal{D}(u)$ is not empty. Indeed, the following are elements of $\mathcal{D}(u)$:
\begin{equation*}
   \bigg(u_+ = \dfrac{u + |u|}{2},\,\,\,u_- = -\dfrac{u - |u|}{2}\bigg), \quad (u_+ = u^2 + 1 + u,\,\,\,u_- = u^2 + 1), \quad (e^u + u, \,\,\,e^u).
\end{equation*}

We propose the following NSFD model for \eqref{eq:5}
\begin{equation}\label{eq:NSFD}
\dfrac{y_i^{k + 1} - y_i^k}{\phi_i(\Delta t, t^k, y^k)} = f_i(t^k, y^k) - y_i^{k + 1}g_i(t^k, y^k) + \varphi_i(\Delta t, t^k, y^k)\bigg(A_i(t^k, y^k) - \dfrac{y_i^{k+1}}{y_i^k}B_i(t^k, y^k)\bigg),
\end{equation}
where
\begin{itemize}
\item $\phi_i(.)$ ($i =1,2, \ldots, n$) are denominator functions satisfying \eqref{eq:6};
\item $A_i(.)$ and $B_i(.)$ ($i =1,2, \ldots, n$) are functions from $\mathbb{R}_+ \times \interior(\mathbb{R}^{n}_+)$ to $\mathbb{R}_+$, which will be determined so that the NSFD scheme is convergent of order $2$;
\item $\varphi_i(\Delta t)$ ($i =1,2, \ldots, n$) are functions of $\Delta t$ that satisfy
\begin{equation}\label{Parameter1}
\begin{split}
&\varphi_i(\Delta t) > 0 \quad \text{for all} \quad \Delta t > 0,\\
&\varphi_i'(0) := \kappa_i > 0. 
\end{split}
\end{equation}
\end{itemize}

First, we investigate the positivity of solutions to the system~\eqref{eq:NSFD}.
\begin{theorem}\label{Theorem1}
If $A_i(t, y)$ and $B_i(t, y)$ ($i = 1, 2, \ldots, n$) satisfy $A_i(t, y) \geq 0$ and $B_i(t, y) \geq 0$ for all $(t, y) \in \mathbb{R}_+ \times \interior(\mathbb{R}^{n}_+)$, then the model \eqref{eq:NSFD} admits the set $\interior(\mathbb{R}^n_+)$ as a positively invariant set for all $\Delta t > 0$. 
In other words, the NSFD method \eqref{eq:NSFD} preserves the positivity of the solution to the dynamical system \eqref{eq:1} for all finite values of the step size. 
\end{theorem}
\begin{proof}
We must prove that $y^k \in \interior(\mathbb{R}^{n}_+)$ for $k > 0$ whenever $y^0 = y(0) \in \interior(\mathbb{R}^{n})$. 
Indeed, we transform \eqref{eq:NSFD} into the explicit form
\begin{equation}\label{eq:NSFDa}
y_i^{k + 1} = \dfrac{y_i^k + \phi_i\,f_i(t^k, y^k) + \phi_i\,\varphi_i\, A_i(t^k, y^k)}{1 + \phi_i\,g_i(t^k, y^k) + \phi_i\,\varphi_i \,B_i(t^k, y^k)/y_i^k} =  \dfrac{(y_i^k)^2 + \phi_i\,y_i^k \,f_i(t^k, y^k) + \phi_i\,\varphi_i \,y_i^k \,A_i(t^k, y^k)}{1 + \phi_i \,y_i^k \,g_i(t^k, y^k) + \phi_i\,\varphi_i\,B_i(t^k, y^k)},
\end{equation}
which implies that $y_i^{k + 1} > 0$ if $y_i^k > 0$. 
Therefore, by mathematical induction, we obtain the desired conclusion. The proof is complete.
\end{proof}

We will now determine the conditions under which the NSFD method \eqref{eq:NSFD} is convergent of order $2$. To this end, let us denote
\begin{equation}\label{eq:9}
v_i(t, y) = \dfrac{\partial F_i}{\partial t}(t, y) + \mathlarger{\mathlarger{\sum}}_{j = 1}^n\dfrac{\partial F_i}{\partial y_j}(t, y)\,F_j(t, y), \quad i = 1, 2, \ldots, n
\end{equation}
\begin{theorem}\label{Theorem2}
Assume that  the following conditions hold for $i = 1, 2, \ldots, n$:
\begin{itemize}
\item $\phi_i(\Delta t, t, y)$ are denominator functions with the property that
\begin{equation}\label{eq:10}
\dfrac{\partial^2 \phi_i}{\partial \Delta t^2}(0,t,y) = 2g_i(t, y)
\end{equation}
for all $(t, y) \in \mathbb{R}_+ \times \interior(\mathbb{R}^{n}_+)$.
\item $\varphi_i(\Delta t)$ satisfy \eqref{Parameter1};
\item $(A_i, B_i)$ is an element of the set $\mathcal{D}(v_i/(2\kappa_i))$.
\end{itemize}
Then,  the NSFD method \eqref{eq:NSFD} satisfies \eqref{eq:4} and its truncated error is $\mathcal{O}(\Delta t^3)$ as $\Delta t \to 0$.
\end{theorem}
\begin{proof}
First, the positivity of the approximate solutions generated by \eqref{eq:NSFD} is a direct consequence of Theorem~\ref{Theorem1}.
To analyze the truncated error, we rewrite \eqref{eq:NSFDa} in the form
\begin{equation*}
   y_i^{k + 1} := V_i(\Delta t, t^k, y_i^k, y^k) = y_i^k + \dfrac{\phi_i\,F_i(t^k, y^k) + \phi_i\,\varphi_i\,\bigl(A_i(t^k, y^k) - B_i(t^k, y^k)\bigr)}{1 + \phi_i\,g_i(t^k, y^k) + \phi_i\,\varphi_i\,B_i(t^k, y^k)/y_i^k}.
\end{equation*}
By some simple manipulations, we obtain
\begin{equation}\label{eq:11}
\begin{split}
&V_i(0, t^k, y_i^k, y^k) = y_i^k,\\
&\dfrac{\partial V_i}{\partial \Delta t}(0, t^k, y_i^k, y^k) = F_i(t^k, y^k),\\
&\dfrac{\partial^2 V_i}{\partial \Delta t^2}(0, t^k, y_i^k, y^k) = F_i(t^k, y^k)\bigg[\dfrac{\partial^2 \phi_i}{\partial \Delta t^2}(0, t^k, y^k) - 2g_i(t^k, y^k)\bigg] + 2\kappa_i\,\bigl(A_i(t^k, y^k) - B_i(t^k, y^k)\bigr).
\end{split}
\end{equation}
On the other hand, applying Taylor's expansion for $y_i(t)$ yields
\begin{equation}\label{eq:12}
\begin{split}
y_i(t^k + \Delta t) &= y_i(t^k) + y_i'(t^k)\Delta t +  \dfrac{1}{2}y_i''(t^k)\Delta t^2\\
&= y_i(t^k) + F_i\bigl(t^k, y(t^k)\bigr)\Delta t 
  + \dfrac{1}{2}v_i\bigl(t^k,y_i(t^k)\bigr)\Delta t^2 + \mathcal{O}(\Delta t^3),
\end{split}
\end{equation}
where $v_i$ is defined in \eqref{eq:9}.
It follows from \eqref{eq:10} and \eqref{eq:11} that
\begin{equation}\label{eq:13}
\begin{split}
y_i^{k + 1} &= V_i(\Delta t, t^k, y_i^k, y^k)\\
&= V_i(0, t^k, y_i^k, y^k) + \dfrac{\partial V_i}{\partial \Delta t}(0, t^k, y_i^k, y^k)\Delta t + \dfrac{1}{2}\dfrac{\partial^2 V_i}{\partial \Delta t^2}(0, t^k, y_i^k, y^k)\Delta t^2 + \mathcal{O}(\Delta t^3)\\
&= y_i^k + F_i(t^k, y^k) \Delta t + 2\kappa_i\,\bigl(A_i(t^k, y^k) - B_i(t^k, y^k)\bigr)\Delta t^2 + \mathcal{O}(\Delta t^3).
\end{split}
\end{equation}
From \eqref{eq:12}, \eqref{eq:13} and  $(A_i, B_i) \in \mathcal{D}(v_i/(2\kappa_i))$, we obtain
\begin{equation*}
y_i^{k + 1} - y_i(t^{k + 1}) = \mathcal{O}(\Delta t ^3).
\end{equation*} 
This is the desired conclusion. The proof is complete.
\end{proof}

\begin{remark}
Theorem \ref{Theorem2} provides a second-order and positivity-preserving NSFD method for \eqref{eq:1} but it does not require the condition \eqref{eq:7}.
A suitable denominator function that satisfies the condition of Theorem \ref{Theorem2} is
\begin{equation}\label{eq:14}
   \phi_i(\Delta t, t, y) = 
\begin{cases}
&\dfrac{e^{2g_i(t, y)\Delta t} - 1}{2g_i(t, y)} \quad \text{if} \quad g_i(t, y) > 0,\\
&\Delta t \quad \text{if} \quad g_i(t, y) = 0.
\end{cases}
\end{equation}
Since $g_i(t, y) \geq 0$, another denominator function can be
\begin{equation}\label{eq:15}
\phi_i(\Delta t, t, y) = g_i(t, y)\Delta t^2 + \Delta t,
\end{equation}
which is simpler that \eqref{eq:14}.
The functions defined in \eqref{eq:14} and \eqref{eq:15} are not bounded as $\Delta t \to \infty$. A denominator function that is bounded as $t \to \infty$ is given by
\begin{equation}\label{eq:15a}
\begin{split}
&\phi(\Delta t, t, y) = \dfrac{\gamma_1(t, y)\Delta t + \gamma_2(t, y)\Delta t^2}{\gamma_3(t, y) + \gamma_4(t, y)\Delta t^3}, \quad \gamma_i(t, y) > 0,\\
&\gamma_1(t, y) = \gamma_3(t, y), \quad \dfrac{\gamma_2(t, y)}{\gamma_3(t, y)} = g(t, y),  \quad m > 2,
\end{split}
\end{equation}
which is suitable when large step sizes are used to observe the behaviour of the dynamical system over long time periods.
\end{remark}
In general, the values of the denominator functions $\phi_i$ are updated at each iteration step. 
However, if the functions $g_i$ are identical constants, that is $g_i(t, y) = g_i$, then the denominator functions do not require an update at each iteration step. 
Assume that the functions $F_i$ ($i = 1, 2, \ldots, n$) have the property that there exists $\alpha_i > 0$ such that
\begin{equation}\label{eq:P1}
F_i(t, y) + \alpha_iy_i \geq 0 \quad \text{for all} \quad t \geq 0, y \in \interior(\mathbb{R}^n_+).
\end{equation}
Many differential equation models have this property (see \cite{Allen, Horvath1, Horvath2, Smith}). 
Hoang \cite{Hoang2023} constructed a generalized NSFD method preserving the positivity of the solutions and the local dynamics of autonomous dynamical systems with the property \eqref{eq:P1}.

Systems that satisfy \eqref{eq:P1} can written in the form
\begin{equation}\label{eq:P2}
y_i' = f_i(t, y) - y_ig_i(t, y), \quad f_i(t, y) = (F_i(t, y) + \alpha_iy_i\big), \quad g_i(t, y) = \alpha_i.
\end{equation}
Therefore, \eqref{eq:NSFD} provides a second-order positivity-preserving NSFD method for which the denominator functions in the form \eqref{eq:14}-\eqref{eq:15a} do not require updating values at each iteration step.

\section{Numerical Simulation of an SIR Epidemic Model}\label{Sec4}
In this section, we perform numerical experiments to support and illustrate the theoretical results. These experiments consider a mathematical epidemiological model.

We consider a modified Susceptible-Infected-Removed (SIR) model \cite{Bailey} as a test problem, which reads
\begin{equation}\label{eq:Model1}
\begin{split}
y_1'(t) & =-\dfrac{b y_1(t) y_2(t)}{y_1(t)+y_2(t)}, \quad y_1(0) > 0,\\
y_2'(t) & =\dfrac{b y_1(t) y_2(t)}{y_1(t)+y_2(t)}- c y_2(t), \quad y_2(0) > 0,\\
y_3'(t) & = c y_2(t),\quad y_3(0) \geq 0,
\end{split}
\end{equation}
where $b$ and $c$ are positive real numbers; $y_1(t)$, $y_2(t)$ and $y_3(t)$ represent the number of susceptible individuals infected individuals and removed individuals at the time $t$, respectively. We refer the readers to \cite{Bailey} for more details of \eqref{eq:Model1}. 

In a recent work \cite{Lemos-Silva}, Lemos-Silva et al.\ applied the Mickens' methodology \cite{Mickens1, Mickens2, Mickens3, Mickens4, Mickens5} to obtain an NSFD model of the following form:
\begin{equation}\label{eq:Model1a}
\begin{split}
&\dfrac{y_1^{k+1}- y_1^k}{\Delta t}=-\dfrac{b y_1^{k+1} y_2^k}{y_1^k + y_2^k}, \\
&\dfrac{y_2^{k+1}- y_2^k}{\Delta t}=\dfrac{b y_1^{k+1} y_2^k}{y_1^k + y_2^k}-c y_2^{k+1}, \\
&\dfrac{y_3^{k+1}- y_3^k}{\Delta t}=c y_2^{k+1}.
\end{split}
\end{equation}
Notably, the exact solution of \eqref{eq:Model1a} has been explicitly determined  in \cite[Theorem 1]{Lemos-Silva}. 
Previously, Bohner et al.\ \cite{Bohner} had proposed a new method for finding the exact solution of \eqref{eq:Model1}, considering not only constant $b$, $c$ but also variable coefficients $b, c \colon \mathbb{R}_+ \to \mathbb{R}_+$.

In this example, we will consider \eqref{eq:Model1} with variable coefficients. Since the total population $N(t) = y_1(t) + y_2(t) + y_3(t)$ is constant for $t \geq t_0$, it is sufficient to consider the first two equations of \eqref{eq:Model1}:
\begin{equation}\label{eq:Model1b}
\begin{split}
y_1'(t) & = -\dfrac{b(t) y_1(t) y_2(t)}{y_1(t)+y_2(t)}, \quad y_1(0) > 0,\\
y_2'(t) & = \dfrac{b(t) y_1(t) y_2(t)}{y_1(t)+y_2(t)}- c(t) y_2(t), \quad y_2(0) > 0.
\end{split}
\end{equation}
We now apply the NSFD method \eqref{eq:NSFD} to \eqref{eq:Model1b}. First, we decompose the right-hand side function of \eqref{eq:Model1b} as follows:
\begin{equation}\label{eq:Model1c}
\begin{split}
&F_1(t, y) =  -\dfrac{b(t) y_1 y_2}{y_1+y_2}, \quad f_1(t, y) = 0, \quad g_1(t, y) = -\dfrac{b(t)y_2}{y_1+y_2},\\
&F_2(t, y) = \dfrac{b(t) y_1 y_2}{y_1 + y_2}- c(t) y_2, \quad f_2(t, y) = \dfrac{b(t) y_1 y_2}{y_1+y_2}, \quad g_2(y) = -c(t)
\end{split}
\end{equation}
By simple calculations, we obtain
\begin{equation}\label{eq:Model1d}
\begin{split}
&v_1(t, y) = -\dfrac{b'y_1y_2}{y_1 + y_2} + \dfrac{b^2y_1y_2^3}{(y_1 + y_2)^3} - \dfrac{b^2y_1^3y_2}{(y_1 + y_2)^3} + \dfrac{bcy_1^2y_2}{(y_1 + y_2)^2},\\
&2\kappa_1A_1(t, y) = -\dfrac{(b')_-y_1y_2}{y_1 + y_2} + \dfrac{b^2y_1y_2^3}{(y_1 + y_2)^3} + \dfrac{bcy_1^2y_2}{(y_1 + y_2)^2},\\
&2\kappa_1B_1(t, y) = \dfrac{(b')_+y_1y_2}{y_1 + y_2} + \dfrac{b^2y_1^3y_2}{(y_1 + y_2)^3}, \quad \big((b')_+, (b')_-\big) \in \mathcal{D}(b'),\\
&v_2(t, y) = \dfrac{b'y_1y_2}{y_1 + y_2} - c'y_2 - \dfrac{b^2y_1y_2^3}{(y_1 + y_2)^3} + \dfrac{b^2y_1^3y_2}{(y_1 + y_2)^3} - \dfrac{bcy_1^2y_2}{(y_1 + y_2)^2} - \dfrac{bcy_1y_2}{y_1 + y_2} + c^2y_2,\\
&2\kappa_2A_2(t, y) = \dfrac{(b')_+y_1y_2}{y_1 + y_2} - (c')_- y_2 + \dfrac{b^2y_1^3y_2}{(y_1 + y_2)^3} + c^2y_2, \quad \big((c')_+, (c)'_-\big) \in \mathcal{D}(c'),\\
&2\kappa_2B_2 = \dfrac{(b')_-y_1y_2}{y_1 + y_2} - (c')_+y_2 - \dfrac{b^2y_1y_2^3}{(y_1 + y_2)^3} - \dfrac{bcy_1^2y_2}{(y_1 + y_2)^2} - \dfrac{bcy_1y_2}{y_1 + y_2}.
\end{split}
\end{equation}
Once the 4-tuple $(\phi_1, \phi_2, \varphi_1, \varphi_2)$ is chosen, \eqref{eq:Model1c} and \eqref{eq:Model1d} define a second-order positivity-preserving NSFD scheme for \eqref{eq:Model1b}.
 In the numerical examples reported below, we will use some NSFD schemes derived from \eqref{eq:Model1c} and \eqref{eq:Model1d}, that utilize the functions $\phi_i$ and $\varphi_i$ given in Table~\ref{Table1}.
\begin{table}[H]
\centering
\caption{NSFD schemes derived from \eqref{eq:Model1c} and \eqref{eq:Model1d} ($\tau > 0$)}\label{Table1}
\begin{tabular}{ccccccccccccccccccccc}
\hline
NSFD scheme & $\phi_1$ & $\phi_2$ & $\varphi_1$ & $\varphi_2$\\
\hline
2dPNSFD1&$g_1(t, y)\Delta t^2 + \Delta t$&$g_2(t, y)\Delta t^2 + \Delta t$&$\Delta t$&$\Delta t$\\[0.2cm]
2dPNSFD2&$\dfrac{e^{2g_1(t, y)\Delta t} - 1}{2g_1(t, y)}$&$\dfrac{e^{2g_2(t, y)\Delta t} - 1}{2g_2(t, y)}$&$\Delta t$&$\Delta t$\\[0.2cm]
2dPNSFD3&$\dfrac{\Delta t + g_1(t, y)\Delta t^2}{1 + \Delta t^3}$&$\dfrac{\Delta t + g_2(t, y)\Delta t^2}{1 + \Delta t^3}$&$1 - e^{-\tau\Delta t}$&$1 - e^{-\tau\Delta t}$\\
\hline
\end{tabular}
\end{table}
Assume that $b(t)$ and $c(t)$ are bounded for $t \geq 0$, that is, there exist $c^*$ and $b^*$ such that
\begin{equation}\label{eq:B1}
\max_{t \geq 0}b(t) = b^* > 0, \quad \max_{t \geq 0}b(t) = c^* > 0.
\end{equation}
Then, \eqref{eq:Model1b} can be rewritten as
\begin{equation}\label{eq:Model1e}
\begin{split}
y_1' & = \bigg(b^*y_1 - \dfrac{b(t) y_1 y_2}{y_1+y_2}\bigg) - b^*y_1,\\
y_2' & = \bigg(\dfrac{b(t) y_1 y_2}{y_1 + y_2 }- c(t) y_2 + c^*y_2\bigg) - c^*y_2.
\end{split}
\end{equation}
Then, the right-hand side function of \eqref{eq:Model1e} can be decomposed in the form
\begin{equation}\label{eq:Model1f}
\begin{split}
&f_1(t, y) = b^*y_1 - \dfrac{b(t) y_1 y_2}{y_1+y_2}, \quad g_1(t, y) = b^*,\\
&f_2(t, y) = \bigg(\dfrac{b(t) y_1 y_2}{y_1 + y_2 }- c(t) y_2 + c^*y_2\bigg), \quad g_2(t, y) = c^*.
\end{split}
\end{equation}

Therefore, once the $4$-tuple $(\phi_1, \phi_2, \varphi_1, \varphi_2)$ is determined, \eqref{eq:Model1d} and \eqref{eq:Model1f} define a second-order positivity-preserving NSFD scheme for \eqref{eq:Model1b}. We consider the following NSFD schemes derived from \eqref{eq:Model1d} and \eqref{eq:Model1f} with the functions $\phi_i$ and $\varphi_i$ in Table~\ref{Table2}
\begin{table}[H]
\centering
\caption{NSFD schemes derived from \eqref{eq:Model1d} and \eqref{eq:Model1f} ($\tau > 0$)}\label{Table2}
\begin{tabular}{ccccccccccccccccccccc}
\hline
NSFD scheme&$\phi_1$&$\phi_2$&$\varphi_1$&$\varphi_2$\\
\hline
2dPNSFD4&$b^*\Delta t^2 + \Delta t$&$c^*\Delta t^2 + \Delta t$&$\Delta t$&$\Delta t$\\[0.2cm]
2dPNSFD5&$\dfrac{e^{2b^*\Delta t} - 1}{2b^*}$&$\dfrac{e^{2c^*\Delta t} - 1}{2c^*}$&$\Delta t$&$\Delta t$\\[0.2cm]
2dPNSFD6&$\dfrac{e^{2b^*\Delta t} - 1}{2b^*}$&$\dfrac{e^{2c^*\Delta t} - 1}{2c^*}$&$1 - e^{-\tau\Delta t}$&$1 - e^{-\tau\Delta t}$\\
\hline
\end{tabular}
\end{table}

Next, we consider \eqref{eq:Model1b} with (see \cite{Bohner})
\begin{equation*}
b(t) = 1/(1 + t), \quad c(t) = 2/(2 + t), \quad y_1(0) = 0.8, \quad y_2(0) = 0.2.
\end{equation*}
The exact solution is given by \cite{Bohner}:
\begin{equation*}
\begin{split}
&y_1(t) = y_1(0)\dfrac{(y_2(0)/y_1(0)) + 1 + t}{\big[(y_2(0)/y_1(0) + 1\big](t + 1)},\\
&y_2(t) = y_2(0)\dfrac{(y_2(0)/y_1(0)) + 1 + t}{\big[(y_2(0)/y_1(0) + 1\big](t + 1)^2}.
\end{split}
\end{equation*}
Note that $b'(t), c'(t) < 0$ for $t \geq 0$. Hence, we choose $(b')_+ = 0$ and $(c')_+ = 0$ in \eqref{eq:Model1d}. On the other hand, $b(t)$ and $c(t)$ satisfy \eqref{eq:B1} with $b^* = 1$ and $c^* = 2$.

We now compare global errors ($err$) at $T = 1$ and rates of convergence ($ROC$) estimated from the NSFD schemes: 2ndNSFD1, 2ndNSFD2, 2ndNSFD3, 2ndNSFD4, 2ndNSFD5 and 2ndNSFD6 and \eqref{eq:Model1a}.
The results are reported in Tables~\ref{Table3}--\ref{Table6}. 
In these tables, the quantities $err$ and $ROC$ are computed similarly to \cite[Example 4.1]{Ascher}.
\begin{equation*}
\begin{split}
&err(\Delta t) = |y_1^N - y_1(t_N)| + |y_2^N - y_2(t_N)|, \quad t_N = 1, \quad \Delta t = \dfrac{1}{N},\\
&ROC = \log_{\bigg(\dfrac{\Delta t_1}{\Delta t_2}\bigg)}\bigg(\dfrac{err(\Delta t_1)}{err(\Delta t_2)}\bigg).
\end{split}
\end{equation*}
Additionally, the graphs of the errors obtained from the second-order NSFD scheme 2ndNSFD2 and the first-order NSFD scheme \eqref{eq:Model1a} with $\Delta t = 0.01$ over $[0, 1]$ are depicted in Figure \ref{Fig:1}.

The results in Tables~\ref{Table2}--\ref{Table5} show that all second-order NSFD schemes 2ndNSFD1, 2ndNSFD2, 2ndNSFD3, 2ndNSFD4, 2ndNSFD5 and 2ndNSFD6 are convergent of order $2$, whereas \eqref{eq:Model1a} is convergent only order $1$. 
Furthermore, the errors of the second-order NSFD schemes depend on the decomposition of the right-hand side function and the chosen $4$-tuple $(\phi_1, \phi_2, \varphi_1, \varphi_2)$. This leads to the problem of optimizing the errors of the second-order NSFD schemes.

\begin{table}[H]
\centering
\caption{Computed errors and ROC of the 2ndNSFD1 and 2ndNSFD2 schemes}\label{Table3}
\begin{tabular}{cccccccccccccccc}
\hline
$\Delta t$&2ndPNSFD1 err&2ndPNSFD1 rate&2ndPNSFD2 err&2ndPNSFD2 rate\\
\hline
0.5&6.559410475124927e-002&&6.094269133987174e-002&&\\
0.25&1.781929796473945e-002& 1.8801&1.337060657340174e-002&2.1884\\
$10^{-1}$&3.275021538910197e-003&1.8487&2.110852617736816e-003&2.0146\\
$10^{-2}$&3.365172105951331e-005&1.9882&1.971819766188876e-005& 2.0296\\
$10^{-3}$&3.366450123387654e-007&1.9998&1.954409423743364e-007&2.0039 \\
$10^{-4}$&3.366496442724909e-009& 2.0000& 1.952614056555113e-009&2.0004\\
$10^{-5}$&3.368197387665362e-011&1.9998&1.954746087218240e-011&1.9995\\
$10^{-6}$&5.451750162421831e-013&1.7909&3.052558206206868e-013&1.8064\\
\hline
\end{tabular}
\end{table}
\begin{figure}[H]
\centering
\vspace*{-3cm}
\includegraphics[width=0.8\textwidth]{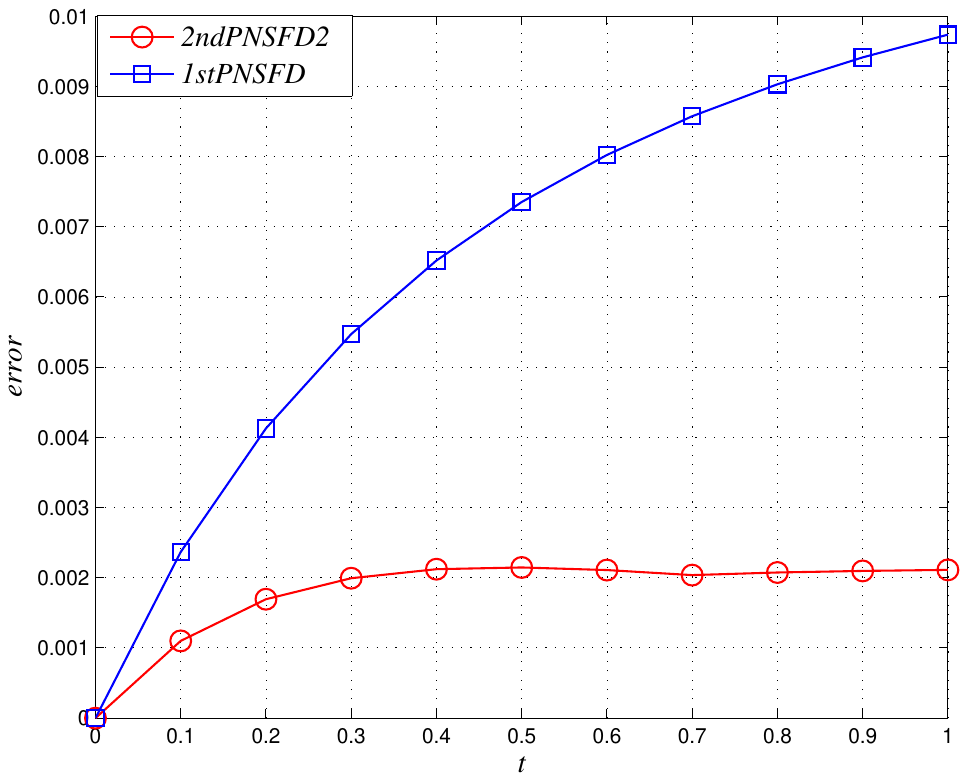}
\vspace*{-4cm}
\caption{Errors obtained from the second-order and first-order NSFD schemes with $\Delta t = 0.01$.}\label{Fig:1}
\end{figure}

\begin{table}[H]
\centering
\caption{Computed errors and ROC of 2ndNSFD3 ($\tau = 5$) and 2ndNSFD4 schemes}\label{Table4}
\begin{tabular}{cccccccccccccccc}
\hline
$\Delta t$&2ndPNSFD3 err&2ndPNSFD3 rate&2ndPNSFD4 err&2ndPNSFD4 rate\\
\hline
0.5&6.944878986451183e-002&&7.902791685990487e-002&\\
0.25&1.644040274307838e-002&2.0787&2.262044634310900e-002&1.8047\\
$10^{-1}$&2.314507864212931e-003&2.1397&3.725237331228468e-003&1.9685\\
$10^{-2}$&2.233206780007102e-005&2.0155&4.118012549277073e-005&1.9565\\
$10^{-3}$&2.220343807701752e-007&2.0025&4.161038818784046e-007&1.9955\\
$10^{-4}$&2.219014763604754e-009&2.0003&4.164809150331017e-009&1.9996\\
$10^{-5}$&2.825792377869618e-011&1.8950&4.781298967859726e-011& 1.9400\\
$10^{-6}$&2.506370111454714e-012&1.0521&8.942110940601822e-012&0.7281\\
\hline
\end{tabular}
\end{table}

\begin{table}[H]
\centering
\caption{Computed errors and ROC of the 2ndNSFD5 and 2ndNSFD6 schemes ($\tau = 5$)}\label{Table5}
\begin{tabular}{cccccccccccccccc}
\hline
$\Delta t$&2ndPNSFD5 err&2ndPNSFD5 rate&2ndPNSFD6 err&2ndPNSFD6 rate\\
\hline
0.5&3.095702263303372e-002&&2.508826597831147e-002&&\\
0.25& 1.271590524117193e-002&1.2836&8.484925721237491e-003&1.5640\\
$10^{-1}$&2.564051674735432e-003&1.7476&1.583635745764422e-003&1.8319\\
$10^{-2}$&2.849100598348309e-005&1.9542 &1.715684442334109e-005&1.9652\\
$10^{-3}$& 2.874830316440535e-007&1.9961&1.728584778509790e-007&1.9967\\
$10^{-4}$& 2.877374796761423e-009&1.9996&1.729869816835539e-009&1.9997\\
$10^{-5}$&2.876071603097330e-011&2.0002& 7.268394219828167e-012& 2.3766\\
$10^{-6}$&5.016959070403004e-013&1.7584&2.456840286768625e-012&0.4711\\
\hline
\end{tabular}
\end{table}
\begin{table}[H]
\centering
\caption{Computed errors and ROC of the first-order NSFD \eqref{eq:Model1a}}\label{Table6}
\begin{tabular}{cccccccccccccccc}
\hline
$\Delta t$&1stNSFD err&1stNSFD rate\\
\hline
0.5&4.272727272727273e-002&& \\
0.25&2.312169312169320e-002&0.8859\\
$10^{-1}$&9.738562091503381e-003&0.9437\\
$10^{-2}$&1.006246214038789e-003&0.9858\\
$10^{-3}$&1.009623162852857e-004&0.9985\\
$10^{-4}$& 1.009962301373735e-005&0.9999\\
$10^{-5}$&1.009996232648192e-006&1.0000\\
$10^{-6}$&1.010000383189214e-007&1.0000\\
\hline
\end{tabular}
\end{table}
Before concluding this section, we will examine the dynamic behavior of the numerical solution generated by the second-order NSFD method using large step sizes. 
To this end, we use the 2ndPNSFD3 scheme to simulate the dynamics of \eqref{eq:Model1b} over $[0, 100]$ and then, compare the numerical solution obtained with those generated by the explicit Euler (first-order) and trapezoidal (second-order) methods (see \cite{Ascher}). The solutions are depicted in Figures \ref{Fig:3} and \ref{Fig:4}. 
Clearly, the 2ndPNSFD3 scheme preserves the dynamical behavior of the continuous model. In contrast, the explicit Euler and trapezoidal schemes produce negative approximations that are negative and differ from the exact solution.
\begin{figure}[htb]
\centering
\vspace*{-3cm}
\includegraphics[width=0.8\textwidth]{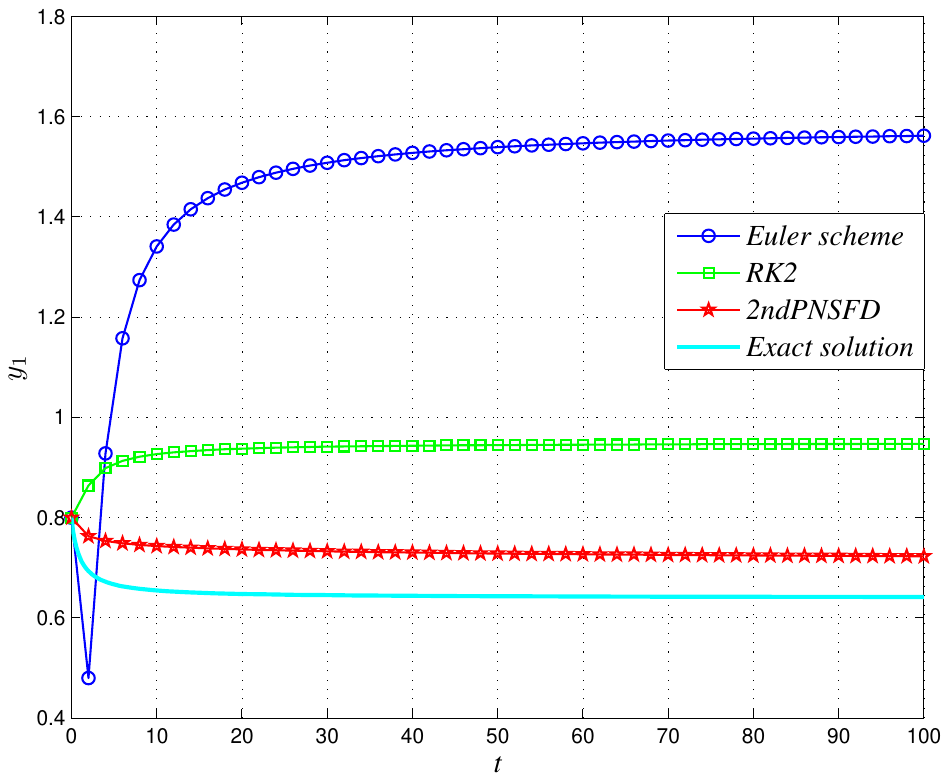}
\vspace*{-4cm}
\caption{Approximate solutions for the $y_1$-components generated by the second-order NSFD scheme and two standard numerical schemes.}\label{Fig:3}
\end{figure}
\begin{figure}[htb]
\centering
\vspace*{-3cm}
\includegraphics[width=0.8\textwidth]{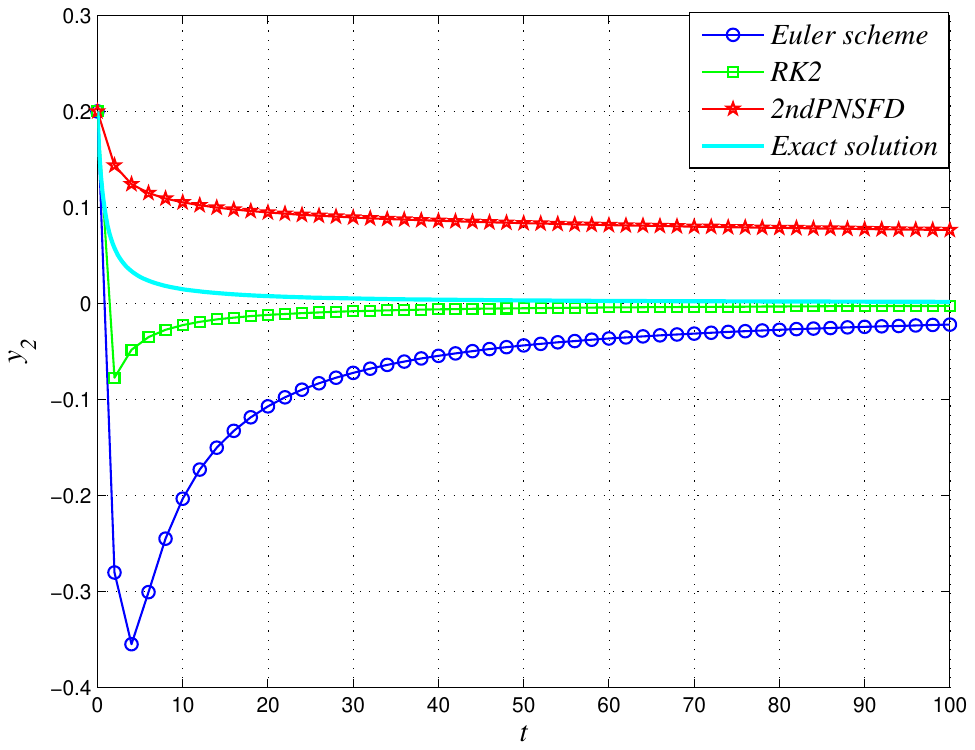}
\vspace*{-4cm}
\caption{Approximate solutions for the $y_2$-components generated by the second-order NSFD scheme and two standard numerical schemes.}\label{Fig:4}
\end{figure}
\begin{remark}
The NSFD schemes \eqref{eq:NSFD1}, \eqref{eq:NSFD2} and \eqref{eq:NSFD3a} are only applicable for the SIR model \eqref{eq:Model1b} when
\begin{equation*}
    \bigl(b(t_k) - c(t_k)\bigr)\,y_1^k \ne c(t_k)\,y_2^k, \quad k \geq 0.
\end{equation*}
\end{remark}

\section{Second-Order Positivity-Preserving NSFD Method Applied to Nonlinear PDEs}\label{Sec5}
In this section, we present an application of the constructed numerical method \eqref{eq:NSFD1} in solving a class of nonlinear PDEs whose solutions are positive.

Many important phenomena and processes arising in mechanics, physics, chemistry, biology, ecology, finance,
environment, etc.\ can be modeled mathematically by nonlinear PDEs (see, e.g., \cite{Allen, Murray1}). 
The solutions of these PDEs often possess essential properties; the most notable of these is the positivity of the solutions. Therefore, constructing numerical methods that preserve the positivity of PDEs is important but not simple in general (see, e.g., \cite{Dias, Mickens1, Mickens2, Mickens3, Mickens4, Mickens5, Mickens6, Mickens7, Patidar1, Patidar2}).

We now consider a class of nonlinear PDEs of the form
\begin{equation}\label{eq:PDE1}
    \dfrac{\partial u(x, t)}{\partial t} + C(u)\,\dfrac{\partial u(x, t)}{\partial x} = D(u)\,\dfrac{\partial^2 u(x, t)}{\partial x^2}+f(u), \quad a \leq x \leq b, \quad 0 \leq t \leq T,
\end{equation}
associated with the boundary conditions
\begin{equation}\label{eq:PDE2}
u(a, t)=a(t), \quad u(b, t)=b(t), \quad 0 \leq t \leq T
\end{equation}
and the initial condition
\begin{equation}\label{eq:PDE3}
u(x, 0) = u_0(x), \quad a \leq x \leq b. 
\end{equation}
In \eqref{eq:PDE1}--\eqref{eq:PDE3}, $C(u)$, $D(u)$, $f(u)$, $a(t)$, $b(t)$ and $u_0(x)$ are functions that satisfy conditions necessary to guarantee unique, positive solutions to the problem \eqref{eq:PDE1}-\eqref{eq:PDE3} on $[a, b] \times [0, T]$. 
The following theorem provides a condition for the solutions of \eqref{eq:PDE1}--\eqref{eq:PDE3} to be positive. 

\begin{theorem}\label{Theorem4}
Assume that $C(u)$, $D(u)$, $f(t, u)$, $a(t)$, $b(t)$ and $u_0(x)$ satisfy conditions that guarantee that the
solutions to the PDE model \eqref{eq:PDE1}-\eqref{eq:PDE3} exist and are unique. Then, $u(t, x) \geq 0$ for $(x, t) \in [a, b] \times [0, T]$ if
\begin{equation*}
\begin{split}
&C(0) \geq 0,\quad D(0) \geq 0,\quad f(0) \geq 0,\\
&a(t) \geq 0, \quad b(t) \geq 0, \quad t \in[0, T],\\
&u_0(x) \geq 0, \quad x \in[a, b].
\end{split}
\end{equation*}
\end{theorem}
\begin{proof}
To prove the theorem, we first use the \textit{method of lines} (MOL) \cite{Ascher, Verwer} to discretize \eqref{eq:PDE1}-\eqref{eq:PDE3} with respect to the space variable. 
To do so, we fix a regular partition $a = x_0 < x_1 < \ldots < x_M = b$ of $[a, b]$ with a step size $\Delta x = (b - a)/M$ and denote by $u_i(t)$ the approximate the value of $u(x, t)$ at $(x_i, t)$ for $i = 0, 1, \ldots, M$.
In these terms, the partial derivatives with respect to $x$ are approximated by finite difference quotients as follows:
\begin{equation}\label{eq:PDE4}
\begin{split}
& \dfrac{\partial u(x, t)}{\partial x} \approx \dfrac{u_i(t) - u_{i-1}(t)}{\Delta x} \\
& \dfrac{\partial^2 u(x, t)}{\partial x^2} \approx \dfrac{u_{i+1}(t)- 2u_i(t) + u_{i-1}(t)}{(\Delta x)^2}, \quad i = 1, 2, \ldots M-1.
\end{split}
\end{equation}
Consequently, we obtain a system of ODEs for ${u_i(t)}\; (i = 1, 2, \ldots, M-1)$:
\begin{equation}\label{eq:PDE5}
\begin{split}
&u_i^{\prime}(t) = -C\bigl(u_i(t)\bigr) \dfrac{u_i(t)-u_{i-1}(t)}{\Delta x} + D\bigl(u_i(t)\bigr) \dfrac{u_{i+1}(t)-2 u_i(t)+u_{i-1}(t)}{(\Delta x)^2}+ f\bigl(u_i(t)\bigr),\\
&u_i(0)=u_0(x_i).
\end{split}
\end{equation}
Note that $u_0(t) = a(t) \geq 0$ and $u_N(t) = b(t) \geq 0$. 
It follows from \eqref{eq:PDE5} that
\begin{equation*}
\left.u_i^{\prime}\right|_{u_i=0}=C(0) u_{i-1}+D(0) \dfrac{u_{i+1}+u_{i-1}}{(\Delta x)^2}+f(0).
\end{equation*}
Therefore, if $C(0), D(0), f(0) \geq 0$ then $u_i^{\prime}\big|_{u_i=0} \geq 0$ for $u_{i+1}, u_{i-1} \geq 0$. By using \cite[Lemma 2]{Horvath2} and \cite[Proposition B.7]{Smith}, we conclude that the set $\mathbb{R}_{+}^{M-1}$ is a positively invariant set of the system \eqref{eq:PDE5}. 

Conversely, the space discretization \eqref{eq:PDE5} is convergent (see \cite{Verwer}), that is, $u_i(t) \rightarrow u\left(t, x_i\right)$ as $\Delta x \rightarrow 0$. Therefore, we conclude that $u(t, x) \geq 0$ for $t \in[0, T]$ and $x \in[a, b]$. 
The proof is complete.
\end{proof}

By using suitable forms of the 3-tuple $C(u), D(u)$ and $f(u)$, we can obtain a huge variety of highly important PDEs models. 
Below, we mention some mathematical models represented by \eqref{eq:PDE1}.
\begin{itemize}
\item If we take $C = 0, D > 0$ and define $f(u) = u(1 - u)(\alpha -u)$ with $0 \leq \alpha \leq 1$, we obtain the Fitzhugh-Nagumo equation
\begin{equation}\label{eq:PDE6}
\dfrac{\partial u(x, t)}{\partial t} = D(u) \dfrac{\partial^2 u(x, t)}{\partial x^2} + u(1 - u)(\alpha - u),
\end{equation}
which arises in population genetics. More details of this model are provided in \cite{Guckenheimer}.

\item In the case $C = 0, D > 0$, and the function $f$ is given by $f(u) = \alpha u + \beta u^m$ with $\alpha, \beta, m \ne 1$, we obtain the Kolmogorov-Petrovskii-Piskunov (KPP) equation:
\begin{equation}\label{eq:PDE7}
\dfrac{\partial u(x, t)}{\partial t} + C(u) \dfrac{\partial u(x, t)}{\partial x}= D(u) \dfrac{\partial^2 u(x, t)}{\partial x^2} + \alpha u + \beta u^m,
\end{equation}
which arises in heat and mass transfer, combustion theory, biology, and ecology. 
More details about the equation can be found in \cite{McKean1}.
An explicit positivity-preserving finite-difference scheme for \eqref{eq:PDE7} was constructed in \cite{Dias}.
\item If $C = 0, D > 0$ and $f(u) = u(1 - u^{\tau})$ with $\tau > 1$, then \eqref{eq:PDE1} generates the Fisher-Kolmogorov equation with applications in biology, see \cite{Murray1}
\begin{equation}\label{eq:PDE8}
\dfrac{\partial u(x, t)}{\partial t}= D(u) \dfrac{\partial^2 u(x, t)}{\partial x^2}+u\left(1-u^\tau\right).
\end{equation}
In the case $C, D > 0$ and $f(u) = \lambda_1 u - \lambda_2 u^2$, where $\lambda_1$ and $\lambda_2$ are both positive, \eqref{eq:PDE1} becomes the Fisher PDE
\begin{equation}\label{eq:PDE9}
\dfrac{\partial u(x, t)}{\partial t} = D(u) \dfrac{\partial^2 u(x, t)}{\partial x^2} + \lambda_1 u - \lambda_2 u^2.
\end{equation}
This equation was considered in \cite{Mickens6}, in which a positivity-preserving NSFD scheme was constructed.

\item Especially, if $D > 0$, $C(u) = u$ and $f(u) = u(1 - u)$, we obtain from \eqref{eq:PDE1} the Fisher
PDE having nonlinear diffusion:
\begin{equation}\label{eq:PDE10}
\dfrac{\partial u(x, t)}{\partial t}+ u \dfrac{\partial u(x, t)}{\partial x}  = D \dfrac{\partial^2 u(x, t)}{\partial x^2} + u(1-u).
\end{equation}
This equation was considered in \cite{Mickens7}, in which a positivity-preserving NSFD scheme was formulated.
\end{itemize}

In this section, we consider \eqref{eq:PDE1}--\eqref{eq:PDE3} with strictly positive solutions:
\begin{equation*}
     u(x, t) > 0 \quad \text{for all} \quad (x, t) \in [a, b] \times [0, T].
\end{equation*}

In order to obtain a positivity-preserving numerical scheme, we apply the second-order NSFD method \eqref{eq:NSFD} for \eqref{eq:PDE5}. For this purpose, we rewrite the system \eqref{eq:PDE5} in the form
\begin{equation}\label{eq:PDE11}
\begin{split}
u_i'(t) &= -\big[C_+(u_i) - C_-(u_i)\big]\dfrac{u_i(t)-u_{i-1}(t)}{\Delta x}\\
&+ \big[D_+(u_i) - D_-(u_i)\big] \dfrac{u_{i+1}(t)-2 u_i(t)+u_{i-1}(t)}{(\Delta x)^2} + f_i(u_i) - u_ig_i(u_i),\\
&(C_+, C_-) \in \mathcal{D}(C), \quad (D_+, D_-) \in \mathcal{D}(D), \quad (f_i, u_ig_i) \in \mathcal{D}(f(u_i)),\\
u_i(0) &= u_0\left(x_i\right),
\end{split}
\end{equation}
or equivalently,
\begin{equation}\label{eq:PDE12}
u_i' := \mathcal{F}_i(u) = F_i(u) - u_iG_i(u),
\end{equation}
where $u = [u_1\,\,\,u_2\,\,\,\ldots\,\,\,u_{M-1}]^\top$ and
\begin{equation*}\label{eq:PDE13}
\begin{split}
&F_i(u) = C_-(u_i)\dfrac{u_{i}}{\Delta x} + C_+(u_i)\dfrac{u_{i-1}}{\Delta x} + D_+(u_i)\dfrac{u_{i+1} + u_{i-1}}{(\Delta x)^2} + D_-(u_i)\dfrac{2u_i}{\Delta x^2} +  f_i(u_i),\\
&G_i(u) = \dfrac{C_+(u_i)}{\Delta x}  + \dfrac{C_-(u_i)}{u_i}\dfrac{u_{i-1}}{\Delta x} + D_-(u_i)\dfrac{u_{i+1} + u_{i-1}}{u_i \Delta x^2} + g_i(u_i).
\end{split}
\end{equation*}
We then obtain a second-order, positivity-preserving numerical scheme for the original PDE model \eqref{eq:PDE1}--\eqref{eq:PDE3}by applying the NSFD method \eqref{eq:NSFD} to \eqref{eq:PDE12}.
\begin{remark}
The NSFD schemes constructed in \cite{Mickens6} for  the Fisher PDE \eqref{eq:PDE9} and in \cite{Dias} for the KPP 
model \eqref{eq:PDE7} are essentially applications of the NSFD methodology to the resulting ODE systems obtained by applying the MOL to the PDE models. Therefore, they only achieve first-order convergence with respect to $\Delta t$.
The second-order positivity-preserving NSFD method \eqref{eq:NSFD} is generally applicable not only for \eqref{eq:PDE12} but also to ODE systems obtained by applying the MOL to PDEs. Thus, it is useful for solving several PDE models with positive solutions.
\end{remark}

\section{Second-Order Positivity-Preserving NSFD Method Applied to BVPs}\label{Sec6}
This section introduces the application of the constructed NSFD method to solving nonlinear BVPs whose solutions are positive.

It is well-known that both linear and nonlinear shooting methods for BVPs lead to solving systems of ODEs \cite{Ascher, Burden}. 
Therefore, the NSFD method \eqref{eq:NSFD} can be used to solve the resulting ODE systems. 
To illustrate this observation, we consider a class of nonlinear BVPs of the form:
\begin{equation}\label{eq:BVP}
u''(t) + \lambda f(u(t)) = 0, \quad 0 \leq t \leq L, \quad u(0) = u(L) = 0,
\end{equation}
which models certain physical problems \cite{Laetsch}, where
\begin{itemize}
\item $f(w) > 0$ for $w > 0$;
\item $\lambda > 0$ is a physical parameter
\end{itemize}
Laetsch \cite{Laetsch} investigated the values of $\lambda$ for which the BVP \eqref{eq:BVP} admits positive solutions, as well as how the behavior of these solutions changes with respect to $\lambda$. 
One of the main results addresses the case in which $f$ is a convex function of $w$ satisfying $f(w) > 0$ for $w > 0$. We refer the reader to \cite{Laetsch} for a more detailed discussion of these findings.

A particularly important special case of the BVP~\eqref{eq:BVP} is the Bratu equation \cite{Bratu}:
\begin{equation*}\label{eq:BVP1}
   u''(t) = -\lambda e^{u(t)}, \quad  0 \leq t \leq 1,
\end{equation*}
subject to the boundary condition
\begin{equation*}
    u(0) = u(1) = 0.
\end{equation*}
This equation has many important theoretical and practical applications, and it is widely used as a benchmark to verify the reliability and efficiency of various approximation methods (see, e.g., \cite{Karamollahi, Tomar} and references therein).

We now assume that the solutions to \eqref{eq:BVP} exist. 
To apply the constructed second-order NSFD method \eqref{eq:NSFD}, we first use the solutions' symmetry to transform the problem of solving \eqref{eq:BVP} into the problem of solving a sequence of ODEs with positive solutions. 
Any solution of \eqref{eq:BVP} is symmetric about the point $t = L/2$, that is $u(t) = u(L/2 - t)$ for $0 \leq t \leq L$ \cite{Laetsch}; hence, we only need to consider \eqref{eq:BVP} on the interval $[0, L/2]$.
 On this interval, it is easy to verify that
\begin{itemize}
\item $u(t) > 0$ for $0 < t \leq L/2$;
\item $u'(t) > 0$ for $0 < t < L/2$;
\item $u'(L/2) = 0$ and therefore, $\max_{0 \leq t \leq L}u(t) = u(L/2)$.
\end{itemize}
As a result, we transform \eqref{eq:BVP} to the following system of ODEs:
\begin{equation}\label{eq:BVP2}
\begin{split}
&u' = v, \quad u(0) = 0,\\
&v' = -\lambda f(u), \quad v(0) = s > 0,
\end{split}
\end{equation}
where the first slope $s$ is determined such that $v'(L/2) = 0$. 
\eqref{eq:BVP2} can be rewritten in the form \eqref{eq:5} as follows:
\begin{equation}\label{eq:BVP3}
\begin{split}
&y_1' = f_1(y_1, y_2) - y_1g_1(y_1, y_2),\\
&y_2' = f_2(y_1, y_2) - y_2g_2(y_1, y_2),
\end{split}
\end{equation}
where
\begin{equation*}\label{eq:BVP4}
\begin{split}
&y_1 = u, \quad y_2 = v,\\
&f_1(y_1, y_2) = y_2, \quad g_1(y_1, y_2) = 0,\\
&f_2(y_1, y_2) = 0, \quad g_2(y_1, y_2) = -\lambda\dfrac{f(y_1)}{y_2}.
\end{split}
\end{equation*}

We obtain positive approximate solutions with second-order accuracy by applying the NSFD method \eqref{eq:NSFD} to \eqref{eq:BVP3}.
The solution of the BVP is obtained by solving a sequence of initial value problems, for which the initial slope is determined via the equation $y_2(s, L/2) = 0$.

\section{Concluding Remarks and Discussions}\label{Sec7}
In this work, we we have proposed a simple and efficient approach for constructing a generalized, second-order, positivity-preserving numerical method for non-autonomous dynamical systems. 
This method is based on a new non-local approximation of the right-hand side function combined with the normalization of denominator functions. 
Notably, the constructed method does not require the strict and indispensable conditions imposed by some well-known second-order positivity-preserving NSFD methods.  
Therefore, a computational implementation is straightforward.

Important applications of the constructed NSFD method are also provided, and numerical experiments are carried out to support and illustrate the theoretical results. As a result, the NSFD scheme for the SIR epidemic model constructed in \cite{Lemos-Silva} has been improved. 
Additionally, applications of the constructed second-order positivity-preserving NSFD method to solving classes of PDEs and BVPs that arise in real-world situations have been introduced and analyzed.

In the near future, we will expand upon the present approach and the results obtained in this work to study the construction of higher-order, dynamically consistent numerical methods for differential equation models with complex dynamics. Additionally, the practical applications of the proposed methods will be of particular interest.

\section*{Acknowledgments}
The first author, Manh Tuan Hoang, wishes to thank the Vietnam Institute for Advanced Study in Mathematics (VIASM) for its financial support and the excellent working conditions. This work was completed while the author was working at the VIASM.

\section*{CRediT authorship contribution statement}
\noindent
\textbf{Manh Tuan Hoang:} Writing - review \& editing, Writing - original draft, Visualization, Validation, Supervision, Software, Resources, Project administration, Methodology, Investigation, Formal analysis, Data curation, Conceptualization, Funding acquisition.

\smallskip\noindent
\textbf{Mathias Ehrhardt:} Writing - review \& editing, Writing - original draft, Visualization, Validation, Supervision, Software, Resources, Project administration, Methodology, Investigation, Formal analysis, Data curation, Conceptualization, Funding acquisition.

\bibliographystyle{amsalpha}

\begin{thebibliography}{999}
%
\bibitem{Alalhareth1}
F. K. Alalhareth, M. Gupta, S. Roy, H. V. Kojouharov, Second-order modified positive and elementary stable nonstandard numerical methods for $n$-dimensional autonomous differential equations, Mathematical Methods in the Applied Sciences 48 (2025) 8037-8057.

\bibitem{Alalhareth2}
F. K. Alalhareth, Higher-order nonstandard finite difference methods for
autonomous differential equations with applications in
mathematical ecology, PhD thesis, The University of Texas at Arlington, 2022.

\bibitem{Alalhareth3}
F. K. Alalhareth, A. C. Mendez, H. V. Kojouharov, A simple model of nutrient recycling and dormancy in a chemostat: Mathematical analysis and a second-order nonstandard finite difference method, Communications in Nonlinear Science and Numerical Simulation 132 (2024) 107940.

%
\bibitem{Allen}
L. J. S. Allen, An Introduction to Mathematical Biology, Pearson, 2006.

\bibitem{Ascher}
U.M. Ascher, L.R. Petzold, Computer Methods for Ordinary Differential Equations and Differential-Algebraic Equations, Society for Industrial and Applied Mathematics, Philadelphia, 1998.

\bibitem{Bailey}
N. T. J. Bailey, The Mathematical Theory of Infectious Diseases and Its Applications, Hafner Press [Macmillan Publishing Co., Inc.], New York, 1975.
%
\bibitem{Bohner}
M. Bohner, S. Streipert, D. F. M. Torres, Exact solution to a dynamic SIR model, Nonlinear Analysis: Hybrid Systems 23(2019) 228-238.

\bibitem{Bratu}
G. Bratu, Sur les \'equations integrales non-lin\'eaires, Bulletin de la Soci\'et\'e Math\'ematique de France 42 (1914) 113-142.

\bibitem{Burden}
R. L. Burden, J. D. Faires, Numerical Analysis, ninth edition, Cengage Learning, 2010.

\bibitem{Conte}
D. Conte, G. Pagano, T. Rold\'an, 
High order nonstandard finite-difference methods, 
Applied Mathematics and Computation 510 (2026) 129681.

\bibitem{Cresson1}
J. Cresson, F. Pierret, Non standard finite difference scheme preserving dynamical properties, Journal of Computational and Applied Mathematics 303 (2016) 15-30.

\bibitem{Cresson2}
J. Cresson, A. Szafra\'nska, Discrete and continuous fractional persistence problems -- the positivity property and applications,  Communications in Nonlinear Science and Numerical Simulation 44 (2017) 424-448.


\bibitem{Dias}
J. E. Mac\'ias-D\'iaz, A. Puri, 
An explicit positivity-preserving finite-difference scheme for the classical Fisher-Kolmogorov-Petrovsky-Piscounov equation, 
Applied Mathematics and Computation 218 (2012) 5829-5837.

\bibitem{Guckenheimer}
J. Guckenheimer, C. Kuehn, Homoclinic orbits of the Fitzhugh-Nagumo equation: The singular limit,
Discrete and continuous dynamical systems Series S 2 (2009) 851-872.

\bibitem{Hirsch}
M. W. Hirsch, S, Smale, R. L. Devaney, Differential Equations, Dynamical Systems, and an Introduction to Chaos, Academic Press, 2013.

\bibitem{Hoang1}
M. T. Hoang, A novel second-order nonstandard finite difference method for solving one-dimensional autonomous dynamical systems,
Communications in Nonlinear Science and Numerical Simulation 114 (2022) 106654.
%
\bibitem{Hoang2}
M. T. Hoang, A novel second-order nonstandard finite difference method preserving dynamical properties of a general single-species model, International Journal of Computer Mathematics 100 (2023) 2047-2062.
%
\bibitem{Hoang3}
M. T. Hoang, High-order nonstandard finite difference methods preserving dynamical properties of one-dimensional dynamical systems, Numerical Algorithms 98 (2025) 219-249.
%
\bibitem{HoangMatthias1}
M. T. Hoang, M. Ehrhardt, A general class of second-order $L$-stable explicit numerical methods for stiff problems, 
Applied Mathematics Letters 149 (2024) 108897.
%
\bibitem{HoangMatthias2}
M. T. Hoang, M. Ehrhardt, A second-order nonstandard finite difference method for a general Rosenzweig-MacArthur predator-prey model, Journal of Computational and Applied Mathematics 444 (2024) 115752.
%
\bibitem{Hoang2023}
M. T. Hoang, A generalized nonstandard finite difference method for a class of autonomous dynamical systems and its applications, Mathematical and Computational Modeling of Phenomena Arising in Population Biology and Nonlinear Oscillations, Contemporary Mathematics 793 (2024) 17-44.
%
\bibitem{Horvath1}
Z. Horv\'ath,  Positivity of Runge-Kutta and diagonally split Runge-Kutta methods, Applied Numerical Mathematics 28(1998) 309-326.

\bibitem{Horvath2}
Z. Horv\'ath,  On the positivity step size threshold of Runge-Kutta methods, Applied Numerical Mathematics 53 (2005) 341-356.

\bibitem{Karamollahi}
N. Karamollahi, G. B. Loghmani, M. Heydari, A computational method to find dual solutions of the one-dimensional Bratu problem, 
Journal of Computational and Applied Mathematics 388 (2021) 113309.

\bibitem{Kojouharov1}
H. V. Kojouharov, S. Roy, M. Gupta, F. Alalhareth, J. M. Slezak,
A second-order modified nonstandard theta method for one-dimensional autonomous differential equations, Applied Mathematics Letters 112 (2021) 106775.



\bibitem{Laetsch}
T. Laetsch, The number of solutions of a nonlinear two point boundary value problem, Indiana University Mathematics Journal 20 (1970)  1-13.

\bibitem{Lemos-Silva}
M. Lemos-Silva, S. Vaz, D. F. M. Torres, Exact solution for a discrete-time SIR model, Applied Numerical Mathematics 207 (2025) 339-347.
%
\bibitem{Mattheij}
R. Mattheij, J. Molenaar, Classics in Applied Mathematics: Ordinary Differential Equations in Theory and Practice, 
Society for Industrial and Applied Mathematics, New York, 2002.

%

\bibitem{McKean1}
H. P. McKean, Application of Brownian motion to the equation of Kolmogorov-Petrovskii-Piskunov, 
Communications on Pure and Applied Mathematics 28 (1975) 323-331.

\bibitem{Mickens1}
R. E. Mickens,
Nonstandard Finite Difference Models of Differential Equations, 
World Scientific, Singapore, 1994.

\bibitem{Mickens2}
R. E. Mickens,
Applications of Nonstandard Finite Difference Schemes,  World Scientific, Singapore, 2000.

\bibitem{Mickens3}
R. E. Mickens, Dynamic consistency: a fundamental principle for constructing nonstandard finite difference schemes for differential equations,  Journal of Difference Equations and Applications 11 (2005) 645-653.

\bibitem{Mickens4}
R. E. Mickens,  Advances in the Applications of Nonstandard Finite Difference Schemes, World Scientific, Singapore, 2005.

\bibitem{Mickens5}
R. E. Mickens, Nonstandard Finite Difference Schemes: Methodology and Applications, World Scientific, 2020.

\bibitem{Mickens6}
R. E. Mickens, Determination of denominator functions for a NSFD scheme for the Fisher PDE with linear advection, 
Mathematics and Computers in Simulation 74 (2007) 190-195.

\bibitem{Mickens7}
R. E. Mickens, 
A Nonstandard Finite Difference Scheme for a Fisher PDE Having Nonlinear Diffusion, 
Computers and Mathematics with Applications 45 (2003) 429-436.

\bibitem{Murray1}
J. G. Murray, Mathematical Biology, I: An Introduction, third ed., Springer, New York, 2002.



\bibitem{Patidar1}
K. C. Patidar, On the use of nonstandard finite difference methods, Journal of Difference Equations and Applications 11(8) (2005) 735-758.

\bibitem{Patidar2}
K. C. Patidar, Nonstandard finite difference methods: recent trends and further developments, Journal of Difference Equations and Applications 22(6) (2016) 817-849


\bibitem{Smith}
H. L. Smith,  P. Waltman, The Theory of the Chemostat: Dynamics of Microbial Competition, Cambridge University Press, 1995.


%
\bibitem{Tomar}
S. Tomar, R. K. Pandey, An efficient iterative method for solving Bratu-type equations, Journal of Computational and Applied Mathematics 357 (2019) 71-84.
\bibitem{Verwer}
J. G. Verwer, J. M. Sanz-Serna, Convergence of method of lines approximations to partial differential
equations, Computing 33 (1984) 297-313.
\bibitem{Wood}
D. T. Wood, H. V. Kojouharov, A class of nonstandard numerical methods for autonomous dynamical systems, 
Applied Mathematics Letters 50 (2015) 78-82.

\bibitem{Wood1}
D. T. Wood, D. T. Dimitrov, H. V. Kojouharov, A nonstandard finite difference method for $n$-dimensional productive-destructive systems, Journal of Difference Equations and Applications  21 (2015) 240-254.


\end{thebibliography}

\end{document}